\documentclass[11pt, a4paper, twoside]{amsart}
\usepackage{amsfonts}
\usepackage{mathrsfs}
\usepackage{amsmath}
\usepackage{amssymb}
\usepackage{fancyhdr}
\usepackage{graphicx}

\allowdisplaybreaks

\renewcommand\Re{\operatorname{Re}}

\newtheorem{theorem}{Theorem}
\newtheorem{lemma}[theorem]{Lemma}

\newtheorem{corollary}[theorem]{Corollary}

\theoremstyle{remark}
\newtheorem{remark}[theorem]{Remark}

\begin{document}

\title{\textbf{Stein-Tomas Restriction Theorem via Spectral Measure on Metric Measure Spaces} }

\author{Xi Chen}

\date{}

\maketitle

\begin{abstract}The Stein-Tomas restriction theorem on Euclidean space says one can meaningfully restrict $\hat{f}$ to the unit sphere of $\mathbb{R}^n$ provided $f \in L^p(\mathbb{R}^n)$ with $1 < p < 2$. This result can be rewritten in terms of the estimates for the spectral measure of Laplacian. Guillarmou, Hassell and Sikora formulated a sufficient condition of the restriction theorem, via spectral measure, on abstract metric measure spaces. But they only proved the result in a special case. The present note aims to give a complete proof. In the end, it will be applied to the restriction theorem on asymptotically conic manifolds.\end{abstract}

\section{Introduction}

Stein \cite{Beijing lecture} and Tomas \cite{Tomas} showed the well-known restriction theorem for Fourier transform to sphere $$\int_{\mathbb{S}^{n - 1}} |\hat{f}|^2 \, d\sigma \leq C \|f\|_{L^p(\mathbb{R}^n)}^2$$ for $p \in [1, 2(n + 1)/(n + 3)].$ This pioneering work stimulates the study of the general restriction problem: given a function $f \in L^p(\mathbb{R}^n)$ with $1 < p < 2$, the question is that "which $p$ will allow one to meaningfully restrict $\hat{f}$ to a subset of $\mathbb{R}^n$?" It is notoriously connected to a few problems in harmonic analysis and dispersive PDEs, including spectral multipliers, Kakeya conjecture, Strichartz estimates etc. We refer the readers to a survey by Tao \cite{Restriction survey} for more information on this subject.

Returning to the Stein-Tomas estimates, one can alternatively write that in terms of the restriction operator to hypersphere, $$R(f)(\xi) = \int_{\mathbb{R}^n} e^{- 2 \pi i x \cdot \xi} f(x) \, dx, \quad |\xi| = 1.$$  Then the restriction theorem reads that $R^\ast R$ is bounded in the sense that $$L^p(\mathbb{R}^n) \longrightarrow L^{p^\prime}(\mathbb{R}^n).$$ Noting that the spectral projection $E_{\sqrt{\Delta}}(\lambda)$ of $\sqrt{\Delta}$ on $\mathbb{R}^n$ is $\mathcal{F}^{-1}(\chi_{B(0, \lambda)}) \mathcal{F}$, we find that the kernel of $R^\ast R$, $(2\pi)^{-n} \int_{|\xi| = 1} e^{i (x - y) \cdot \xi} \, d\xi$, is indeed $dE_{\sqrt{\Delta}}(1)$, where $dE_{\sqrt{\Delta}}(\lambda)$ is the spectral measure. Therefore, one may rewrite the restriction theorem as  $$ \|dE_{\sqrt{\Delta}}(\lambda)\|_{L^p \rightarrow L^{p^\prime}} (= \lambda^{n(2/p - 1) - 1} \|dE_{\sqrt{\Delta}}(1)\|_{L^p \rightarrow L^{p^\prime}}) \leq C \lambda^{n(2/p - 1) - 1},$$ provided $p \in [1, 2(n + 1)/(n + 3)].$

Due to this relation between restriction estimates and spectral measure, it is legitimat to seek an abstract restriction theorem under certain assumption on spectral measure. Guillarmou, Hassell and Sikora \cite{Guillarmou-Hassell-Sikora} formulated following theorem on metric measure spaces. \begin{theorem}\label{abstract restriction theorem}Consider an $n$-dimensional metric measure space $(X, d, \mu)$ and let $L$ be an abstract positive self-adjoint operator on $L^2(X, d\mu)$. We have the restriction estimates \begin{equation}\label{eqn:abstract restriction}  \|dE_{\sqrt{L}}(\lambda)\|_{L^p \rightarrow L^{p^\prime}}  \leq C \lambda^{n(1/p - 1/p') - 1},\end{equation} provided $p \in [1, 2(n + 1)/(n + 3)],$ if the following conditions hold: \begin{enumerate}\renewcommand{\labelenumi}{$($\roman{enumi}$)$}
\item Factorization of spectral measure: the spectral measure of $\sqrt{L}$ can be factorized as \begin{equation}\label{eqn:factorization}dE_{\sqrt{L}}(\lambda) = \frac{1}{2\pi} P(\lambda)P(\lambda)^\ast, \end{equation} where $P(\lambda)$ is defined on $L^2(X)$;
\item Operator partition of unity: for each $\lambda$, there are are uniformly bounded operators $Q_i$ on $L^2(X)$ such that \begin{equation}\label{eqn:partition of unity}Id = \sum_{i = 1}^{N} Q_i(\lambda);\end{equation} 
\item Spectral measure estimates: given a number $\lambda$ in the essential spectrum of $\sqrt{L}$, if $\mu$ is a number very close to $\lambda$, say $\mu/\lambda \in (1 - \delta, 1 + \delta)$ with small positive number $\delta$, then for any $1 \leq i \leq N(\lambda)$ we have \begin{equation}\label{eqn:spectral measure estimates}\bigg|Q_i(\lambda) \frac{d^k}{d \mu^k}\Big(dE_{\sqrt{L}}(\mu)\Big) Q_i^\prime(\lambda)\bigg| \leq C \mu^{n - 1 - k} (1 + \mu w(z, z'))^{-(n - 1)/2 + k},
\end{equation}provided a nonnegative function $w(z, z')$ on $X \times X$.
\end{enumerate}\end{theorem}

\begin{remark}\label{remark}There exist manifolds with boundary $(M, g)$ admitting the factorization \eqref{eqn:factorization}. For example, asymptotically hyperbolic manifolds and asmyptotically conic manifolds. On the other hand, the function $w$ is usually the geodesic distance function on $M$; in the meantime, the operators $Q_i$ from partition of unity \eqref{eqn:partition of unity} can be thought of as (micro)-localization for spectral measure in \eqref{eqn:spectral measure estimates} when there are pairs of conjugate points on $M$. Since arising conjugate points will damage the global estimates, it is necessary to localize the targeted operator.\end{remark}

However, the given proof in \cite{Guillarmou-Hassell-Sikora} deals only with the easier case when the spectral measure yields $$\bigg| \bigg( \frac{d^j}{d \lambda^j} \Big(dE_{\sqrt{L}}(\lambda)\Big) \bigg) (z, z') \bigg| \leq C \lambda^{n - 1 - j} \big(1 + \lambda w(z, z')\big)^{-(n-1)/2 + j}.$$ In terms of the language of manifolds, it only gives the restriction theorem on  manifolds without conjugate points. But it is not obvious how to complete the proof in the gerneral case where there is a partition of unity.

The aim of the present work is to give a complete proof of Theorem \ref{abstract restriction theorem}. For the application to asymptotically conic manifolds, Guillarmou, Hassell and Sikora \cite{Guillarmou-Hassell-Sikora} proved the spectral measure estimates \eqref{eqn:spectral measure estimates} in accordance with the partition of unity \eqref{eqn:partition of unity}. Therefore, we will obtain Stein-Tomas restriction theorem on such manifolds, once Theorem \ref{abstract restriction theorem} is justified. Additionally, Hassell and the author \cite{Chen-Hassell2} showed better spectral measure estimates on asymptotically hyperbolic manifolds than \eqref{eqn:spectral measure estimates}. Using that together with Kunze-Stein phenomenon, the authors established a stronger restriction theorem.  

My PhD supervisor Andrew Hassell noticed the gap in their joint paper while we were working on the spectral measure on asymptotically hyperbolic manifolds. Colin Guillarmou also shared some useful thoughts for the present proof. I am indebted to them for their help. This work is supported by a Riemann fellowship and finished during my visit to the Riemann Center for Geometry and Physics at the Leibniz Universit\"{a}t Hannover. I hereby thank the Institut f\"{u}r Analysis for their hospitality.

\section{Proof of Theorem}

First of all, we use a standard $TT^\ast$ argument to reduce the restriction theorem to \begin{equation}\label{eqn:microlocalized restriction}\|Q_i(\lambda)\phi(\sqrt{L}/\lambda)dE_{\sqrt{L}}(\lambda)Q_i^\prime(\lambda)\|_{L^p \rightarrow L^{p'}} \leq C \lambda^{n(1/p - 1/p') - 1}, \end{equation} where $\phi$ is a smooth function supported on $(1 - \delta, 1 + \delta)$ and equal to $1$ on $(1 - \delta/2, 1 + \delta/2)$ provided $\delta$ is a very small number. This idea was used in  \cite[p. 913]{Guillarmou-Hassell-Sikora}, but the key difference here is the cut-off function $\phi$. Assume \eqref{eqn:microlocalized restriction} for the moment. It yields that \begin{eqnarray*}\|Q_i(\lambda)dE_{\sqrt{L}}(\lambda)Q_i^\prime(\lambda)\|_{L^p \rightarrow L^{p'}} &=& \|Q_i(\lambda)\phi(\sqrt{L}/\lambda)dE_{\sqrt{L}}(\lambda)Q_i^\prime(\lambda)\|_{L^p \rightarrow L^{p'}} \\&\leq& C \lambda^{n(1/p - 1/p') - 1},\end{eqnarray*} for all $\lambda$. The factorization \eqref{eqn:factorization} and the $TT^\ast$ trick gives $$\|Q_i(\lambda)P(\lambda)\|_{L^2 \rightarrow L^{p'}}  \leq C \lambda^{n(1/2 - 1/p') - 1/2},$$ for all $\lambda$. By the partition of unity \eqref{eqn:partition of unity}, we sum them up over $i$ and deduce the global estimates $$\|P(\lambda)\|_{L^2 \rightarrow L^{p'}}  \leq C \lambda^{n(1/2 - 1/p') - 1/2},$$ for all $\lambda$. The restriction estimates \eqref{eqn:abstract restriction} follow from applying the factorization \eqref{eqn:factorization} and the $TT^\ast$ trick again.

It now suffices to prove \eqref{eqn:microlocalized restriction} at the endpoints $1$ and $2(n + 1)/(n + 3)$. We only need to show the latter as the former is immediate from the spectral measure estimates \eqref{eqn:spectral measure estimates}. For the endpoint $2(n + 1)/(n + 3)$, we consider the analytic family $Q_i(\lambda)\phi(\sqrt{L}/\lambda)\chi_+^{a}(\lambda - \sqrt{L})Q_i'(\lambda)$ with $a \in \mathbb{C}$. For $\Re a > -1$, $\chi_+^a$ is defined by $$
\chi_+^a = \frac{x_+^a}{\Gamma(a + 1)} \quad \mbox{with} \quad x_+^a = \bigg\{ \begin{array}{r@{\quad \mbox{if} \quad}l}x^a & x \geq 0\\ 0 & x < 0\end{array} ,  \quad \Re a > -1,
$$ where $\Gamma$ is the gamma function. Using the identity
$$\frac{d}{dx}\chi_+^a = \chi_+^{a - 1},\mbox{when $\Re a > 0$},$$ we extend $\chi_+^a$ to the entire complex $a$-plane. Since $\chi_+^0 = H(x)$, we have $\chi_+^{-1} = \delta_0$, and more generally $\chi_+^{-k} = \delta_0^{(k - 1)}$. Therefore,
\begin{equation}
\chi_+^0(\lambda - P) = E_{P}\big((0, \lambda]\big) \quad
\mbox{and} \quad \chi_+^{-k}(\lambda - P) =
\Big(\frac{d}{d\lambda}\Big)^{k - 1} dE_{P}(\lambda).\label{chi-k}\end{equation} 
This family of distributions obey following convolution identity \cite[p. 86]{Hormander1} 
\begin{equation}
\label{eqn:convolution id} \chi_+^\mu \ast \chi_+^\nu = \chi_+^{\mu + \nu + 1}, \quad \mbox{for any $\mu,\nu \in \mathbb{C},$ } \end{equation} as well as following inequality, \begin{lemma}[\cite{Guillarmou-Hassell-Sikora}]\label{GHS lemma}Given $k \in \mathbb{N}$, suppose $- k < a < b < c$ and $0 < \theta < 1$ with $b = \theta a + (1 - \theta)c$. For any compactly supported complex valued function $f \in C^{k - 1}(\mathbb{R})$, we have \begin{equation}\|\chi_+^{b + is} \ast f \|_\infty \leq C (1 + |s|) e^{\pi |s| / 2} \|\chi_+^a \ast f\|_\infty^\theta \|\chi_+^c \ast f\|_\infty^{1 - \theta},\end{equation} for all $s \in \mathbb{R}$.\end{lemma}

Compared with \cite{Guillarmou-Hassell-Sikora}, the key innovation is to include the cut-off function $\phi$. It has no effect on the spectral measure, but allows one to localize the $\chi_+^a$ operator.

Back to the analytic family $Q_i(\lambda)\phi(\sqrt{L}/\lambda)\chi_+^{a}(\lambda - \sqrt{L})Q_i'(\lambda)$, the $L^2 \rightarrow L^2$ boundedness on the line $\{is : s \in \mathbb{R}\}$ $$\|Q_i(\lambda)\phi(\sqrt{L}/\lambda)\chi_+^{is}(\lambda - \sqrt{L})Q_i'(\lambda)\|_{L^2 \rightarrow L^2} \leq C e^{\pi |s|/2} $$ is clear. On the other hand, we have to show the $(1, \infty)$ boundedness on the line $\{-(n + 1)/2 + is : s \in \mathbb{R}\}$, namely \begin{lemma}\label{lemma:1-infinty} $$\|Q_i(\lambda)\phi(\sqrt{L}/\lambda)\chi_+^{-(n + 1)/2 + is}(\lambda - \sqrt{L})Q_i'(\lambda)\|_{L^1 \rightarrow L^\infty} \leq C (1 + |s|) e^{\pi |s|/2} \lambda^{(n - 1)/2}.$$\end{lemma} If Lemma \ref{lemma:1-infinty} is true, \eqref{eqn:microlocalized restriction} will be proved by an application of Stein's complex interpolation \cite{interpolation} between the two lines.

\begin{proof}[Proof of Lemma \ref{lemma:1-infinty}]

 By \eqref{eqn:convolution id},  one may rewrite \begin{eqnarray*}\chi_+^{-(n + 1)/2 + is} = \left\{\begin{array}{ll}\chi_+^{-3/2 -is} \ast \chi_+^{-k} (\lambda) & n = 2k \\  \chi_+^{-2 - is} \ast \chi_+^{-k} (\lambda) & n = 2k + 1\end{array}\right. \end{eqnarray*}
 
When $n = 2k + 1$, we firstly write the operator in the following convolution form. \begin{eqnarray*}
\lefteqn{\phi(\sqrt{L}/\lambda)\chi_+^{-(n + 1)/2 + is}(\lambda - \sqrt{L})}\\ &=& \int \phi(\sigma/\lambda) \chi_+^{-2 - is} \ast \chi_+^{-k} (\lambda - \sigma) dE_{\sqrt{L}}(\sigma)\, d\sigma \\ &=& \iint \phi(\sigma/\lambda) \chi_+^{- 2 - is}(\alpha) \chi_+^{- k}(\lambda - \alpha - \sigma) dE_{\sqrt{L}}(\sigma)\, d\sigma d\alpha \\ &=& \lambda^{-k - is}\iint \chi_+^{- 2 - is}(\alpha) \chi_+^{- k}(1 - \alpha - \sigma) \phi(\sigma) dE_{\sqrt{L}}(\lambda\sigma)\,d\sigma d\alpha
\\ &=& \lambda^{-k - is}\int \chi_+^{- 2 - is}(\alpha)  \frac{d^{k - 1}}{d \sigma^{k - 1}}\Big(\phi(\sigma) dE_{\sqrt{L}}(\lambda\sigma)\Big)\bigg|_{\sigma = 1 - \alpha}\, d\alpha. \end{eqnarray*} Note the targeted operator $\phi(\sqrt{L}/\lambda)\chi_+^{-(n + 1)/2 + is}(\lambda - \sqrt{L})$ is now of convolution form  $f_1 \ast f_2 (1)$, where $f_1(\alpha) = \chi_+^{- 2 - is}(\alpha)$ and  $$f_2(\alpha) =  \lambda^{-k - is} \frac{d^{k - 1}}{d \alpha^{k - 1}}\Big(\phi(\alpha) dE_{\sqrt{L}}(\lambda\alpha)\Big).$$

Then we microlocalize the operator and apply Lemma \ref{GHS lemma}.\begin{eqnarray*}
\lefteqn{\bigg\|Q_i(\lambda)\phi(\sqrt{L}/\lambda)\chi_+^{-(n + 1)/2 + is}(\lambda - \sqrt{L})Q_i^\prime(\lambda)\bigg\|_\infty}\\ &=& \bigg|\lambda^{-k - is}\int \chi_+^{- 2 - is}(\alpha) Q_i(\lambda) \frac{d^{k - 1}}{d \sigma^{k - 1}}\Big(\phi(\sigma) dE_{\sqrt{L}}(\lambda\sigma)\Big)\bigg|_{\sigma = 1 - \alpha} Q_i^\prime(\lambda)\,  d\alpha \bigg|\\ &\leq& C (1 + |s|) e^{\pi |s| / 2}\lambda^{-k} \sup_{\Lambda}\bigg|\int \chi_+^{- 1}(\alpha) Q_i(\lambda) \frac{d^{k - 1}}{d \sigma^{k - 1}}\Big(\phi(\sigma) dE_{\sqrt{L}}(\lambda\sigma)\Big)\bigg|_{\sigma = \Lambda - \alpha} Q_i^\prime(\lambda)\,  d\alpha\bigg|^{1/2}\\ & & \quad\quad\,\sup_{\Lambda}\bigg|\int \chi_+^{- 3}(\alpha) Q_i(\lambda) \frac{d^{k - 1}}{d \sigma^{k - 1}}\Big(\phi(\sigma) dE_{\sqrt{L}}(\lambda\sigma)\Big)\bigg|_{\sigma = \Lambda - \alpha} Q_i^\prime(\lambda)\,  d\alpha\bigg|^{1/2}.\end{eqnarray*} 
We now plug in \eqref{chi-k} and get
\begin{eqnarray*}
\lefteqn{\bigg\|Q_i(\lambda)\phi(\sqrt{L}/\lambda)\chi_+^{-(n + 1)/2 + is}(\lambda - \sqrt{L})Q_i^\prime(\lambda)\bigg\|_{\infty}}\\ &\leq& C (1 + |s|) e^{\pi |s| / 2}  \lambda^{-(n - 1)/2} \sup_{\Lambda}\bigg| Q_i(\lambda) \frac{d^{k - 1}}{d \sigma^{k - 1}}\Big(\phi(\sigma) dE_{\sqrt{L}}(\lambda\sigma)\Big)\bigg|_{\sigma = \Lambda} Q_i^\prime(\lambda)\bigg|^{1/2}\\ & & \quad\quad \quad\quad\,\sup_{\Lambda}\bigg| Q_i(\lambda) \frac{d^{k + 1}}{d \sigma^{k + 1}}\Big(\phi(\sigma) dE_{\sqrt{L}}(\lambda\sigma)\Big)\bigg|_{\sigma = \Lambda } Q_i^\prime(\lambda)\bigg|^{1/2}. \end{eqnarray*}  Noting the support of $\phi$ is in $(1 - \delta, 1 + \delta)$, we can invoke the spectral measure estimates \eqref{eqn:spectral measure estimates}\begin{eqnarray*}
\lefteqn{\bigg\|Q_i(\lambda)\phi(\sqrt{L}/\lambda)\chi_+^{-(n + 1)/2 + is}(\lambda - \sqrt{L})Q_i^\prime(\lambda)\bigg\|_\infty}\\ &\leq&  C (1 + |s|) e^{\pi |s| / 2} \lambda^{(n - 1)/2} \sup_{\Lambda\in \text{supp}\phi}\sum_{j = 0}^{k - 1}\bigg| \Lambda^{n - 1 - j} (1 + \lambda \Lambda d(z, z'))^{- (n - 1)/2 + j}\bigg|^{1/2}\\ & & \quad\quad \quad\quad\,\sup_{\Lambda \in \text{supp}\phi}\sum_{j = 0}^{k + 1}\bigg|\Lambda^{n - 1 - j}(1 + \lambda\Lambda d(z, z'))^{- (n - 1)/2 + j}\bigg|^{1/2}\\ &\leq&C (1 + |s|) e^{\pi |s| / 2} \lambda^{(n - 1)/2},\end{eqnarray*} which proves the lemma in the odd-dimensional case.

When $n = 2k$, the proof is identical, apart from using the other expression of the operator, which is\begin{eqnarray*}
\lefteqn{\phi(\sqrt{L}/\lambda)\chi_+^{-(n + 1)/2 + is}(\lambda - \sqrt{L})} \\ &&= \lambda^{- (n - 1)/2 - is}\int \chi_+^{- 3/2 - is}(\alpha)  \frac{d^{k - 1}}{d \sigma^{k - 1}}\Big(\phi(\sigma) dE_{\sqrt{L}}(\lambda\sigma)\Big)\bigg|_{\sigma = 1 - \alpha}\, d\alpha.\end{eqnarray*}We run the same argument and get \begin{eqnarray*}
\lefteqn{\bigg\|Q_i(\lambda)\phi(\sqrt{L}/\lambda)\chi_+^{-(n + 1)/2 + is}(\lambda - \sqrt{L})Q_i^\prime(\lambda)\bigg\|_\infty}\\ &=& \bigg|\lambda^{- (n - 1)/2 - is}\int \chi_+^{- 3/2 - is}(\alpha) Q_i(\lambda) \frac{d^{k - 1}}{d \sigma^{k - 1}}\Big(\phi(\sigma) dE_{\sqrt{L}}(\lambda\sigma)\Big)\bigg|_{\sigma = 1 - \alpha} Q_i^\prime(\lambda)\,  d\alpha \bigg|\\  &\leq& C (1 + |s|) e^{\pi |s| / 2}\lambda^{-(n - 1)/2} \sup_{\Lambda}\bigg| Q_i(\lambda) \frac{d^{k - 1}}{d \sigma^{k - 1}}\Big(\phi(\sigma) dE_{\sqrt{L}}(\lambda\sigma)\Big)\bigg|_{\sigma = \Lambda} Q_i^\prime(\lambda)\bigg|^{1/2}\\ & & \quad\quad \quad\quad\,\sup_{\Lambda}\bigg| Q_i(\lambda) \frac{d^{k }}{d \sigma^{k}}\Big(\phi(\sigma) dE_{\sqrt{L}}(\lambda\sigma)\Big)\bigg|_{\sigma = \Lambda } Q_i^\prime(\lambda)\bigg|^{1/2} \\ &\leq&C (1 + |s|) e^{\pi |s| / 2} \lambda^{(n - 1)/2} \sup_{\Lambda\in \text{supp}\phi}\sum_{j = 0}^{k - 1}\bigg| \Lambda^{n - 1 - j} (1 + \lambda \Lambda d(z, z'))^{- (n - 1)/2 + j}\bigg|^{1/2}\\ & & \quad\quad \quad\quad\,\sup_{\Lambda \in \text{supp}\phi}\sum_{j = 0}^{ k }\bigg|\Lambda^{n - 1 - j}(1 + \lambda\Lambda d(z, z'))^{- (n - 1)/2 + j}\bigg|^{1/2}\\ &\leq&C (1 + |s|) e^{\pi |s| / 2} \lambda^{(n - 1)/2}.\end{eqnarray*} Now the proof is complete.\end{proof}

\section{Application to Asymptotically Conic Manifolds}

Consider an $n$-dimensional compact manifold with boundary $(M, g)$. The metric $g$ is smooth in the interior $M^\circ$ and takes the form $$g = \frac{dx^2}{x^4} + \frac{h(x)}{x^2},$$ in a collar neighborhood near $\partial M$, where $x$ is a boundary defining function for $M$ and $h(x)$ is a smooth one-parameter family of metrics on $\partial M$.  The Laplacian is denoted by $\Delta_g$.

Such a manifold admits the factorization \eqref{eqn:factorization}, where the factor $P(\lambda)$ is the family of Poisson operators. They map $L^2(\partial M)$ into the null space of $\Delta_g - \lambda^2$ and satisfy \eqref{eqn:factorization}. We refer the reader to Hassell and Vasy \cite{Hassell-Vasy}. On the other hand, Guillarmou, Hassell and Sikora \cite{Guillarmou-Hassell-Sikora} established the spectral measure estimates \eqref{eqn:spectral measure estimates} on $M$ in accordance with the partition of unity \eqref{eqn:partition of unity}. Here the $w$ is the geodesic distance function $d(z, z')$ and $Q_i$ is a family of pseudodifferential operators which microlocalize the spectral measure as mentioned in Remark \ref{remark}.  

Therefore, we obtain following Stein-Tomas restriction theorem on asymptotically conic manifolds, \begin{corollary}Let $(M, g)$ be an $n$-dimensional non-trapping asymptotically conic manifold. We have $$\|dE_{\sqrt{\Delta_g}}(\lambda)\|_{L^p(M) \rightarrow L^{p^\prime}(M)}  \leq C \lambda^{n(1/p - 1/p') - 1},$$ for all $\lambda > 0$.\end{corollary}

\begin{flushleft}
\vspace{0.3cm}\textsc{Mathematical Sciences
Institute\\Australian National University\\Canberra 0200, Australia}



\emph{E-mail address}: \textsf{herr.chenxi@outlook.com}

\end{flushleft}

\end{document}